\documentclass{amsart}
\usepackage{amssymb}
\usepackage{multirow}
\usepackage{hyperref}
\usepackage{url}
\usepackage[all]{xy}

\newtheorem{thm}{Theorem}

\newtheorem{lem}[thm]{Lemma}
\newtheorem{cor}[thm]{Corollary}
\newtheorem{prop}[thm]{Proposition}
\newtheorem{conj}[thm]{Conjecture}

\newtheorem{ques}[thm]{Question}

\theoremstyle{definition}
\newtheorem{defn}[thm]{Definition}
\newtheorem{rem}[thm]{Remark}
\newtheorem{exa}[thm]{Example}

\numberwithin{thm}{section}

\newfont{\cyrr}{wncyr10}
\def\Sh{\mbox{\cyrr Sh}}

\def\Z{\mathbf{Z}}
\def\Q{\mathbf{Q}}
\def\F{\mathbf{F}}
\def\R{\mathbf{R}}

\def\Zp{\Z_p}

\def\Ql{\Q_\ell}
\def\Fp{\F_p}

\def\E{\tilde{E}}
\def\cE{\mathcal{E}}
\def\O{\mathcal{O}}

\def\cC{\mathcal{C}}

\def\cN{\mathcal{N}}

\def\J{{J}}

\def\l{\mathfrak{l}}
\def\q{\mathfrak{q}}

\def\n{\mathfrak{n}}
\def\p{\mathfrak{p}}

\def\Hom{\mathrm{Hom}}
\def\Gal{\mathrm{Gal}}
\def\cork{\mathrm{corank}}
\def\rk{\mathrm{rank}}
\def\ord{\mathrm{ord}}

\def\Aut{\mathrm{Aut}}
\def\Sel{\mathrm{Sel}}

\def\Frob{\mathrm{Frob}}

\def\ur{\mathrm{ur}}

\def\ST{\Sh\mathrm{T}}

\def\image{\mathrm{image}}
\def\Spec{\mathrm{Spec}}

\def\Tr{\mathrm{Trace}}

\def\ord{\mathrm{ord}}

\def\too{\longrightarrow}
\def\map#1{\;\xrightarrow{#1}\;}
\def\isom{\xrightarrow{\sim}}
\def\hookto{\hookrightarrow}
\def\onto{\twoheadrightarrow}
\def\dirsum#1{\underset{#1}{\textstyle\bigoplus}}

\def\bmu{\boldsymbol{\mu}}

\def\Xset{\mathcal{X}}
\def\H{\mathcal{H}}

\def\Hfppf{H_{\mathrm{fppf}}}

\title{Selmer companion curves}

\author{Barry Mazur}
\address{Department of Mathematics, 
Harvard University,
Cambridge, MA 02138, 
USA}
\email{\href{mailto:mazur@math.harvard.edu}{mazur@math.harvard.edu}}
\author{Karl Rubin}
\address{Department of Mathematics, 
UC Irvine,
Irvine, CA 92697, 
USA}
\email{\href{mailto:krubin@math.uci.edu}{krubin@math.uci.edu}}

\thanks{This material is based upon work supported by the 
National Science Foundation under grants DMS-1065904 and DMS-0968831.}

\begin{document}

\begin{abstract}
We say that two elliptic curves $E_1, E_2$ over a number field $K$ are 
{\em $n$-Selmer companions} for a positive integer $n$ if for every quadratic character 
$\chi$ of $K$, there is an isomorphism $\Sel_n(E_1^\chi/K) \cong \Sel_n(E_2^\chi/K)$ 
between the $n$-Selmer groups of the quadratic twists $E_1^\chi$, $E_2^\chi$.
We give sufficient conditions for two elliptic curves to be $n$-Selmer companions, 
and give a number of examples of non-isogenous pairs of companions.
\end{abstract}

\maketitle

\section{Introduction}

There has recently been much interest in the relative densities of $p$-Selmer ranks of 
elliptic curves in families consisting of all quadratic twists of a given elliptic curve 
$E$ over a fixed number field $K$. There are 
\begin{itemize} 
\item 
conjectures about these relative densities \cite{poonenrains},
\item 
some theorems for $p=2$  for elliptic curves over $K=\Q$ with full $2$-torsion rational 
\cite{heathbrown,swinnertondyer} 
\item  
some theorems for $p=2$ for elliptic curves over arbitrary number fields $K$ 
with other restrictions on the rational $2$-torsion \cite{klagsbrun,kmr}
\end{itemize}
The results here are still fragmentary, but the general sense of these conjectures and results 
is that the relative densities of the sets of characters $\chi$ such that $E^{\chi}$ has a fixed 
$p$-Selmer rank (the characters being ordered appropriately \dots) are expected to 
depend only on a few of the basic features of the elliptic curve $E$ and number field $K$.
For example, if $p=2$ and $\Gal(K(E[2])/K) \cong S_3$, then 
these densities depend only on a single parameter, a rational number $\delta(E,K) \in [0,1]$ 
that we call the {\em disparity} \cite{kmr}.  Moreover, if $K$ has a real embedding then $\delta(E,K)=1/2$, 
so  for such $E$ and $K$ these densities are absolute constants.

At the present we have no results at all giving rank densities of
$p$-Selmer groups of quadratic twist families when $p > 2$.  For example,
fix a number field $K$, a prime number $p$, 
and an elliptic curve $E$ over $K$.  We expect that the ranks of the $p$-Selmer groups 
in the family of all quadratic twists of $E$ over $K$ are unbounded, but we can
prove this for no triple $(K, p, E)$ with $p > 2$.  Nonetheless, we have precise expectations 
\cite{poonenrains} for the statistics governing $p$-rank densities in this context.

In view of this type of constancy of densities of Selmer ranks, let us consider the 
Selmer rank function itself, rather than its statistics, and formulate the following 
inverse question.   
For every positive integer $n$, let 
$\Sel_n(E/K)$ denote the $n$-Selmer group of $E$ over $K$
(see Definition \ref{seldef}), and for every quadratic character 
$\chi$ of $K$, we denote by $E^\chi$ the quadratic twist of $E$ by $\chi$.

\begin{ques}
For a fixed prime $p$ and number field $K$, how much information 
about $E$ can be read off from the function $\chi \mapsto \dim_{\Fp}\hskip-2.5pt\Sel_p(E^\chi/K)$?  
\end{ques}

This leads us to the following definition.

\begin{defn}
We will say that two elliptic curves $E_1, E_2$ over $K$ are 
{\em $n$-Selmer companions} if for every quadratic character $\chi$ of $K$ 
there is a group isomorphism
$$
\Sel_n(E_1^\chi/K) \cong \Sel_n(E_2^\chi/K).
$$ 
\end{defn}

If $n = \prod p_i^{k_i}$, then $E_1$ and $E_2$ are $n$-Selmer companions if and only 
if they are $p_i^{k_i}$-Selmer companions for every $i$.  For this reason we will 
restrict in what follows to the case where $n$ is a prime power.

If $E_1$ and $E_2$ are isogenous over $K$ by an isogeny of prime degree $\ell$, 
then so are $E_1^\chi$ and $E_2^\chi$ for every $\chi$, and 
the induced maps $\Sel_n(E_1^\chi/K) \to \Sel_n(E_1^\chi/K)$ are isomorphisms for 
every $n$ prime to $\ell$.  Thus in this case $E_1$ and $E_2$ are $n$-Selmer companions 
for every $n$ prime to $\ell$.  For this reason we are mainly 
interested in non-isogenous $n$-Selmer companions, although the case of isogenies 
of degree dividing $n$ is still interesting.

Our main result is Theorem \ref{comp} below, which gives sufficient 
conditions for a pair of curves to be $p^k$-Selmer companions for a prime power $p^k$.  

For example, we will show (see \S\ref{exas}) that the following 
pairs of non-isogenous curves, given with their labels from \cite{cremona}, are 
$p^k$-Selmer companions for $p^k$ as listed, over every number field $K$ (and 
in particular over $\Q$):
$$
\renewcommand\arraystretch{1.2}
\begin{array}{|c|c|l|}
\hline
~p^k~ & \text{label} & \text{curve} \\
\hline
\hline
\multirow{2}{5pt}{4}
  & 1242L1 &y^2+xy+y = x^3-x^2+x+1 \\
  & 1242K1 &y^2+xy+y = x^3-x^2-1666739-2448131309 \\
\hline
\hline
\multirow{2}{5pt}{5}
  & 676B1 &y^2 = x^3 + x^2 - 4x - 12 \\
  & 676E1 &y^2 = x^3 - 28561x + 1856465 \\
\hline
\hline
\multirow{2}{5pt}{7}
  & 1026N1 &y^2 + xy + y = x^3 - x^2 - 34601x - 2468631 \\
  & 1026O1 &y^2 + xy + y = x^3 - x^2 - 4241x + 107353 \\
\hline
\hline
\multirow{2}{5pt}{9}
  & 6555D1 &y^2 + y = x^3 - x^2 + 59335x + 3888371 \\
  & 6555E1 &y^2 + y = x^3 - x^2 - 33735977475x - 2384987222304844 \\
\hline
\end{array} 
$$

A natural open problem is to give necessary and sufficient conditions for $p$-companionship, 
with a prime $p$.  
We have not found any example of a pair of $p$-Selmer companions having different conductors, 
or having different Galois action on their $p$-torsion points.  In \S\ref{converse} 
we discuss some related questions, including partial converses to Theorem \ref{comp}.

It also seems natural to expect that a given elliptic curve can have only finitely many 
$p$-Selmer companions, even if we allow the prime $p$ to vary.  It is not difficult to show that 
for a given $E_1/K$ and fixed $p$, there are only finitely many $E_2/K$ satisfying the hypotheses of 
Theorem \ref{comp}, because all such curves must have the same set of primes of bad reduction.
See Proposition \ref{justfin}.

Finally, much of what we do in this paper applies more generally to the Bloch-Kato Selmer groups 
\cite{BlochKato} attached to a motive.  The subtlest problem in the general case, as in the case 
of elliptic curves, is to understand the local condition at $p$ in the definition of the 
Bloch-Kato $p^k$-Selmer group, and how it changes under twists.  
It would be of interest to try to express that local condition only in
terms of finite congruence information related to the motive.

\subsection*{Notation}
If $K$ is a field, we let $\bar{K}$ denote a fixed algebraic closure of $K$, 
and $G_K = \Gal(\bar{K}/K)$.  Let $\Xset(K) := \Hom(G_K,\{\pm1\})$ be the 
group of quadratic characters of $K$.

If $E$ is an elliptic curve over $K$ and $m$ is a positive integer, then 
$E[m]$ will denote the kernel of multiplication by $m$ on $E$.

If $K$ is a local field and $E$ is an elliptic curve over $K$, then 
we will say that $E$ has potentially good reduction if  
$E$ has good reduction over some finite extension $F$ of $K$, 
{\em including} the case where $E$ has good reduction over $K$, i.e., $F = K$.
Similarly we will say that $E$ has potentially multiplicative reduction if  
$E$ has multiplicative reduction over a finite extension $F$ of $K$, 
including the case where $E$ has multiplicative reduction over $K$.

\section{Elliptic curves over local fields}

For this section suppose that $K$ is a finite extension of $\Ql$ 
for some rational prime $\ell$, 
and $E$ is an elliptic curve over $K$ with potentially multiplicative reduction.
Let $\ord_K : K^\times \onto \Z$ denote the valuation on $K$.

\begin{defn}
\label{spldef}
There is a unique character $\psi = \psi_{E/K}\in\Xset(K)$ such that the twist $E^\psi$
has split multiplicative reduction (see for example \cite[\S1.12]{serre1972}).  
We will call $\psi$ the {\em splitting character} of $E/K$.

The next proposition is due to Tate \cite[Theorem 1]{tatemr}.
\end{defn}

\begin{prop}[Tate \cite{tatemr}]
\label{tateunif}
There is a Tate parameter $q = q_E \in K^\times$ with $\ord_K(q) < 0$, and an isomorphism 
$$
\tau_{E/K} : \bar{K}^\times/q^\Z \isom E(\bar{K}).
$$
If $E$ has split multiplicative reduction, then $\tau_{E/K}$ is 
$G_{K}$-equivariant. 
In general, we have 
$$
\tau_{E/K}(u^\sigma) = \psi(\sigma)\tau_{E/K}(u)^\sigma
$$
for every $\sigma\in G_{K}$, 
where $\psi$ is the splitting character of $E/K$.
\end{prop}

Let $\bmu_m$ denote the group of $m$-th roots of unity in $\bar{K}$, and 
$$
\epsilon_m : G_K \to \Aut(\bmu_m) \isom (\Z/m\Z)^\times
$$
the mod $m$ cyclotomic character.

\begin{defn}
For $m > 0$, 
define the canonical subgroup 
$$
\cC_{E/K}[m] := \tau_{E/K}(\bmu_m) \subset E[m].
$$
\end{defn}

\begin{lem}
\label{Cchar}
There is a $G_K$-isomorphism $\cC_{E/K}[m] \cong \bmu_m \otimes \psi_{E/K}$, i.e.,
$\cC_{E/K}[m]$ is a cyclic group of order $m$ on which $G_K$ acts by the character 
$\epsilon_m \psi_{E/K}$.
\end{lem}

\begin{proof}
This is immediate from Proposition \ref{tateunif}.
\end{proof}

Let $j(E)$ denote the $j$-invariant of $E$.

\begin{lem}
\label{sameC}
Suppose $E_1, E_2$ are elliptic curves over $K$ with potentially multiplicative reduction, 
$p$ is a prime, and $p \nmid \ord_K(j(E_1))$.
Suppose further that either 
$p \ne 2$, 
or 
$E_1$ and $E_2$ have multiplicative reduction.

Then for every $k \ge 1$ and every $G_K$-isomorphism $\lambda : E_1[p^k] \isom E_2[p^k]$, we have
$$\lambda(\cC_{E_1/K}[p^k]) = \cC_{E_2/K}[p^k].$$
\end{lem}

\begin{proof}
Let $q_i$ be the Tate parameter of $E_i/K$, and $\psi_i$ 
the splitting character.  
Let $\tau_i = \tau_{E_i/K}$ be the map of Proposition \ref{tateunif}.  

Fix a $G_K$-isomorphism $\lambda : E_1[p^k] \isom E_2[p^k]$.
By Proposition \ref{tateunif}, for every $\sigma \in G_K$ and $\beta \in (\bar{K}^\times/q_i^\Z)[p^k]$, 
we have 
\begin{equation}
\label{bidef}
\psi_i(\sigma)\tau_i(\beta)^\sigma - \tau_i(\beta) = \tau_i(\beta^\sigma/\beta) \in \tau_i(\bmu_{p^k}) = \cC_{E_i/K}[p^k].
\end{equation}
Define 
$$
B_i = \{z^\sigma - z : \text{$z \in E_i[p^k]$, $\sigma \in G_K$, and $\psi_1(\sigma) = \psi_2(\sigma) = 1$} \}.
$$ 
Then by \eqref{bidef} we have $B_i \subset \cC_{E_i/K}[p^k]$, and since $\lambda$ is $G_K$-equivariant 
we have $\lambda(B_1) = B_2$.  

Choose $\gamma \in \bar{K}$ with $\gamma^p = q_1$.  
We have $\ord_K(q_1) = -\ord_K(j(E_1))$, so $\ord_K(q_1)$ is prime to $p$ and 
$[K(\gamma):K] = p$.  
We will show that there is a $\sigma \in G_K$ such that $\gamma^\sigma \ne \gamma$ and 
$\psi_1(\sigma) = \psi_2(\sigma) = 1$.

\medskip\noindent{\em Case 1: $E_1$ and $E_2$ have multiplicative reduction.}
Since $\ord_K(q_1)$ is prime to $p$, $K(\gamma)/K$ is ramified, so there is a $\sigma$ 
in the inertia group $I_K \subset G_K$ such that $\gamma^\sigma \ne \gamma$.  If 
$E_1$ and $E_2$ have multiplicative reduction, then $\psi_1$ and $\psi_2$ are unramified, 
so $\psi_1(\sigma) = \psi_2(\sigma) = 1$.  

\medskip\noindent{\em Case 2: $p > 2$.}
Since $[K(\gamma):K] = p$, we have $\gamma \notin K(\bmu_p)$.  
Thus we can choose $\sigma_0 \in G_{K(\bmu_p)}$ such that $\gamma^{\sigma_0} \ne \gamma$. 
Let $\sigma = \sigma_0^2$.   
Since $p \ne 2$ and $\sigma_0$ fixes $\gamma^{\sigma_0}/\gamma \in \bmu_p$, 
we have that $\gamma^{\sigma}/\gamma = (\gamma^{\sigma_0}/\gamma)^2  \ne 1$.
Since $\psi_1$ and $\psi_2$ are quadratic characters, we also have 
$\psi_1(\sigma) = \psi_2(\sigma) = 1$.

\medskip
Now choose $\beta \in \bar{K}$ with $\beta^{p^k} = q_1$.  Then $\beta^\sigma/\beta$ 
is a primitive $p^k$-th root of unity, so $\tau_1(\beta^\sigma/\beta)$ generates 
$\cC_{E_1/K}[p^k]$, and by \eqref{bidef} we have that $\tau_1(\beta^\sigma/\beta) \in B_1$.
Therefore $B_1 = \cC_{E_1/K}[p^k]$, so 
$$
\lambda(\cC_{E_1/K}[p^k]) = \lambda(B_1) = B_2 \subset \cC_{E_2/K}[p^k].
$$ 
The final inclusion must be an equality because both groups have order $p^k$ by Lemma \ref{Cchar}.
\end{proof}

\section{Main theorem}

Suppose for this section that $K$ is a number field.  
Fix a prime $p$ and a power $p^k$ of $p$, $k \ge 1$.
Our main result is the following.

\begin{thm}
\label{comp}
Suppose $E_1$ and $E_2$ are elliptic curves over $K$.  Let 
$S_i$ be the set of primes of $K$ where $E_i$ has 
potentially multiplicative reduction.
Let $m = p^{k+1}$ if $p \le 3$, and $m = p^k$ if $p>3$.  
Suppose further that:
\begin{enumerate}
\item
there is a $G_K$-isomorphism $E_1[m] \cong E_2[m]$,
\item
$S_1 = S_2$, 
\item
for all $\l \in S_1 = S_2$, the isomorphism of (i) sends $\cC_{E_1/K_\l}[m]$ to $\cC_{E_2/K_\l}[m]$,
\item
for every $\p$ of $K$ above $p$, either 
\begin{itemize}
\item
$\p \in S_1 = S_2$, or
\item
$k=1$, $E_1$ and $E_2$ have good reduction at $\p$, and the ramification degree $e(\p/p)$ is less than $p-1$.
\end{itemize}
\end{enumerate}
Then for every finite extension $F$ of $K$, and every $\chi \in \Xset(F)$, 
there is a canonical isomorphism 
$$
\Sel_{p^k}(E_1^\chi/F) \cong \Sel_{p^k}(E_2^\chi/F).
$$
In particular $E_1$ and $E_2$ are $p^k$-Selmer companions over every number field containing $K$.  
\end{thm}

A slightly stronger version of Theorem \ref{comp} will be proved in \S\ref{proofsect} below.
We first give some remarks, consequences and examples.

\begin{rem}
The proof of Theorem \ref{comp} will show that $\Sel_{p^k}(E_1^\chi/F)$, $\Sel_{p^k}(E_2^\chi/F)$ 
are actually equal inside $H^1(F,E_1^\chi[p^k]) = H^1(F,E_2^\chi[p^k])$, 
where we use the isomorphism $E_1[p^k] \cong E_2[p^k]$ to identify 
$H^1(F,E_1^\chi[p^k])$ with $H^1(F,E_2^\chi[p^k])$.
See Definition \ref{seldef} and the remarks following it.\end{rem}

\begin{rem}
\label{comp+}
A careful reading of the proof will show that when $p = 3$, if $E_1$ and $E_2$ have no primes of 
additive reduction of Kodaira type $\rm{II}$, $\rm{IV}$, $\rm{II}^*$, or $\rm{IV}^*$, 
then we can take $m = 3^k$ 
instead of $3^{k+1}$ in the hypotheses of Theorem \ref{comp}.  See Remark \ref{rem3.2}.
In particular if $E$ is semistable, we can take $m=3^k$ instead of $m=3^{k+1}$.
\end{rem}

\begin{cor}
Suppose $K'$ is a number field with a unique prime $\p$ above $p$.  
Suppose $E_1$ and $E_2$ are elliptic curves over $K'$ 
with potentially multiplicative reduction at $\p$, and with potentially 
good reduction at all primes different from $\p$.  
Let $m = p^{k+1}$ if $p \le 3$, and $m = p^k$ if $p>3$.
Then $E_1$ and $E_2$ are $p^k$-Selmer companions over every number field $K$ containing 
$K'(E_1[m],E_2[m])$.
\end{cor}

\begin{proof}
Fix any group isomorphism $\lambda : E_1[m] \to E_2[m]$ that 
takes $\cC_{E_1/K'_\p}[m]$ to $\cC_{E_2/K'_\p}[m]$.  If $K$ 
is a field containing $K'(E_1[m],E_2[m])$, then $\lambda$ is $G_K$-equivariant 
because $G_K$ acts trivially on both sides.  We have $S_1 = S_2$ is the set of all 
primes of $K$ above $\p$, so all hypotheses of Theorem \ref{comp} are satisfied.
\end{proof}

\begin{rem}
One can also twist elliptic curves by characters of order $\ell > 2$.  In this case the 
twist $E^\chi$ is an abelian variety of dimension $\ell-1$, with an action of 
the group of $\ell$-th roots of unity.  One can study the $\n$-Selmer groups $\Sel_\n(E^\chi/K)$, 
where $\n$ is an ideal of the cyclotomic field of $\ell$-th roots of unity, 
and ask when two elliptic curves $E_1$, $E_2$ have the property that 
$\Sel_\n(E_1^\chi/K) \cong \Sel_\n(E_2^\chi/K)$ for every character 
$\chi$ of $G_K$ of order $\ell$.  An analogue of Theorem \ref{comp} holds in this case, 
with a similar proof.
\end{rem}

\section{Examples}
\label{exas}
Suppose we are given elliptic curves $E_1$, $E_2$ over $K$.  
We next introduce some tools to prove in practice that 
$E_1$ and $E_2$ satisfy the hypotheses of Theorem \ref{comp}.
Let $m = p^k$ or $p^{k+1}$ as in Theorem \ref{comp}.

\begin{rem}
The most difficult hypothesis of Theorem \ref{comp} to check is (i), the existence 
of a $G_K$-isomorphism $E_1[m] \cong E_2[m]$.  
When $m \le 5$ this can be done by computing explicitly the (genus zero) family of 
all elliptic curves $E$ over $K$ with $E[m] \cong E_1[m]$, and checking whether 
$E_2$ belongs to this family.  

For larger $m$ the problem is more subtle.
A necessary condition for the existence of such an isomorphism is that 
for every prime $\l$ of $K$ where $E_1$ and $E_2$ have good reduction, 
the traces of Frobenius $\Frob_\l$ acting on $E_1[m]$ and $E_2[m]$ are the same.  
In some cases this necessary condition can be turned into a sufficient condition 
(see for example \cite{krausoesterle}).  We will take a more brute force 
approach, and for candidate curves that satisfy the necessary condition on traces 
for many primes, we will simply construct an isomorphism directly.  

We will discuss both of these approaches in more detail in Appendix \ref{sm8pf}.
\end{rem}

It is natural to ask how common are pairs of non-isogenous curves $E_1$, $E_2$ with 
$E_1[m] \cong E_2[m]$ as $G_K$-modules.  For some history of this question, 
see \cite{mazur-isog,krausoesterle,kanischanz}.  
Along these lines, we raise the following question.

\begin{ques}
Is there an integer $N$ (depending only on the degree $[K:\Q]$) such that for every $m > N$, there are {\em no} 
pairs of non-isogenous elliptic curves $E_1$, $E_2$ over $K$ such that $E_1[m]$ is $G_K$-isomorphic 
to $E_2[m]$?
\end{ques}

If $E$ is an elliptic curve, let $j(E)$ denote its $j$-invariant.

\begin{rem}
It is well-known (see for example \cite{tatealg}) that 
an elliptic curve $E$ has potentially multiplicative reduction at 
a prime $\l$ if and only if $\ord_\l(j(E)) < 0$.  Thus $S_1 = S_2$ if and only if $j(E_1)$ and 
$j(E_2)$ have the same primes in their denominators.
This makes it very easy to check hypothesis (ii) of Theorem \ref{comp}.

Hypothesis (iii) can often be verified by using Lemma \ref{sameC}.
\end{rem}

\begin{exa}
\label{ex10}
Consider the first example of the Introduction, with $K = \Q$, $p^k=4$, and 
\begin{align*}
&E_1 : y^2+xy+y = x^3-x^2+x+1, \\
&E_2 : y^2+xy+y = x^3-x^2-1666739-2448131309.
\end{align*} 
These are the curves 1242L1 and 1242K1, respectively, in \cite{cremona}.
The following proposition can be proved using the method of \S\ref{method2} of Appendix \ref{sm8pf}, 
using computations in Sage \cite{sage} and PARI/GP \cite{pari}.  
Its proof will be sketched at the end of \S\ref{method2}.

\begin{prop}
\label{samemod8}
There is a $G_\Q$-isomorphism $E_1[8] \cong E_2[8]$.
\end{prop}

Computing the $j$-invariants gives
\begin{equation}
\label{ex1j}
j(E_1) = \frac{3^3\cdot 7^3}{2\cdot 23}, \quad j(E_2) = \frac{3\cdot 987697^3}{2^{49}\cdot 23}
\end{equation}
so $S_1 = S_2 = \{2, 23\}$.  Combined with Proposition \ref{samemod8}, this 
shows that Theorem \ref{comp}(i,ii,iv) are satisfied.
For $\ell \in \{2,23\}$ both $E_1$ and $E_2$ have multiplicative reduction at $\ell$, 
and $\ord_\ell(j(E_1))$ is odd, so Lemma \ref{sameC} shows that Theorem \ref{comp}(iii) is satisfied.

Thus all hypotheses of Theorem \ref{comp} 
are satisfied, and we conclude that $E_1$ and $E_2$ are $4$-Selmer companions over every number field.
\end{exa}

The third and fourth examples of the introduction, with $p^k = 7$ and $9$, are proved in exactly the 
same way, using \S\ref{method2} to construct a $G_\Q$-isomorphism $E_1[m] \cong E_2[m]$.  
For the $p^k=9$ example, we use Remark \ref{rem3.2} below so that (since $E_1$ and $E_2$ are semistable) 
we can take $m= 9$ in Theorem \ref{comp} instead of $m = 27$.

\begin{exa}
\label{ex11}
Consider the second example of the Introduction, with $K = \Q$, $p^k=5$, and 
\begin{align*}
E_1 : y^2 &= x^3 + x^2 - 4x - 12 \\
E_2 : y^2 &= x^3 - 28561x + 1856465
\end{align*} 
These are the curves $676B1$ and $676E1$, respectively, in \cite{cremona}.
We will prove the following proposition in \S\ref{method1} of Appendix \ref{sm8pf}.

\begin{prop}
\label{samemod5}
There is a $G_\Q$-isomorphism $E_1[5] \cong E_2[5]$.
\end{prop}

Both $E_1$ and $E_2$ have conductor $676 = 2^2 \cdot 13^2$, $j(E_1) = -208$, and 
$j(E_2) = 1168128$.  Thus $E_1, E_2$ have good reduction at $5$ and potentially good 
reduction everywhere, so $S_1 = S_2$ is the empty set and all hypotheses of
Theorem \ref{comp} are satisfied.
We conclude that $E_1$ and $E_2$ are $5$-Selmer companions over every number field.
\end{exa}

The strategy of Example \ref{ex11}, along with the method described in \S\ref{method1} of 
Appendix \ref{sm8pf}, works in exactly the same way to 
give many more examples of $p$-Selmer companions over $\Q$ with $p = 2, 3$ or $5$.

\section{Elliptic curves over local fields, continued}

For this section let $K$ be a finite extension of $\Ql$ or of $\R$.  
Fix a rational prime power $p^k$.  
For every elliptic curve $E/K$, let $\kappa_E = \kappa_{E/K}$ denote 
the Kummer map 
$$
\kappa_{E/K} : E(K) \too H^1(K,E[p^k]).
$$
For the proof of Theorem \ref{comp} we need to understand the image of $\kappa_E$.  
We will consider several cases.

\subsection{Potentially good reduction}

\begin{lem}
\label{12.2}
Suppose that either $K$ is archimedean, or $K$ is nonarchimedean of residue characteristic 
different from $p$ and $E$ has bad, potentially good reduction.
\begin{enumerate}
\item
If $p \le 3$, then $\image(\kappa_{E})$ 
is represented by the cocycles
$$
\{c_t : t \in E[p^{k+1}], p^k t \in E(K)\},
$$
where  $c_t$ is defined by $c_t(\sigma) = t^\sigma - t$ for $\sigma \in G_K$.
\item
If $p > 3$, or if $v$ is archimedean and $p > 2$, then $\image(\kappa_E) = 0$.
\end{enumerate}
\end{lem}

\begin{proof}
We consider the archimedean and nonarchimedean cases separately.

\medskip\noindent{\em Case 1: $K$ is archimedean.}
In this case (ii) is clear, because $G_K$ has order $1$ or $2$ so $H^1(K,E[p^k]) = 0$.

For every $p$, the map $E(K)[p] \to E(K)/p^kE(K)$ is surjective, so 
$\image(\kappa_E) = \kappa_E(E(K)[p])$.
If $s \in E(K)[p]$, then $\kappa_E(s)$ is represented by the cocycle 
$c_t$ for (every) $t \in E[p^{k+1}]$ with $p^k t = s$.  This proves (i) in this case.

\medskip\noindent{\em Case 2: $K$ is nonarchimedean of residue characteristic 
different from $p$ and $E$ has bad, potentially good reduction.}
Consider the filtration
$$
E_1(K) \subset E_0(K) \subset E(K)
$$
where $E_0(K)$ is the subgroup of points with nonsingular reduction, and $E_1(K)$ is the 
subgroup of points that reduce to zero \cite[Chapter VII]{silverman}.  Then 
$E_1(K)$ is a pro-$\ell$-group, where $\ell$ is the residue characteristic of $K$. 
Since $E$ has additive reduction, $E_0(K)/E_1(K)$ is isomorphic to the additive group 
of the residue field of $K$, 
also an $\ell$-group.
The description of $E(K)/E_0(K)$ in \cite{tatealg} shows that the $p$-part of $E(K)/E_0(K)$ 
is killed by $p$, and is trivial if $p > 3$.  Hence $E(K)/p^kE(K) = 0$ if $p > 3$, and 
$E(K)[p] \to E(K)/p^kE(K)$ is surjective for every $p$.  Therefore 
$\image(\kappa_E) = 0$ if $p > 3$, and (as in Case 1) 
$\image(\kappa_E) = \kappa_E(E(K)[p])$ is represented by the cocycles 
$\{c_t : t \in E[p^{k+1}], p^k t \in E(K)\}$ if $p \le 3$.
This completes the proof.
\end{proof}

\begin{rem}
\label{rem3.2}
Suppose $p = 3$.  If $K$ is nonarchimedean and the Kodaira type of the reduction of $E$ is not 
$\rm{II}$, $\rm{IV}$, $\rm{II}^*$, or $\rm{IV}^*$, then \cite{tatealg} 
shows that $E(K)/E_0(K)$ has order prime to $3$.  
Thus in those cases we have $\image(\kappa_E) = 0$, just as when $p > 3$.  
This allows us to take $m = 3^k$ instead of $m = 3^{k+1}$ in Theorem \ref{comp} 
in this case (see Remark \ref{comp+}).
\end{rem}

\begin{lem}
\label{samered}
Suppose $K$ is nonarchimedean of residue characteristic not dividing $m$, and $E_1$, $E_2$ 
are elliptic curves over $K$ with potentially good reduction.  
If $m \ge 3$ and $E_1[m] \cong E_2[m]$ as $G_{K}$-modules, then either both 
$E_1$ and $E_2$ have good reduction, or both have bad reduction. 
\end{lem}

\begin{proof}
This follows directly from \cite[Corollary 2(b) of Theorem 2]{serre-tate}.
\end{proof}

\subsection{Potentially multiplicative reduction}

Recall that if $E$ has potentially multiplicative reduction, then Proposition \ref{tateunif} 
gives a parameter $q$ and an isomorphism $\tau_{E/K} : \bar{K}^\times/q^\Z \to E(\bar{K})$.

\begin{lem}
\label{3.4}
Suppose $E/K$ has split multiplicative reduction.  
Then $\image(\kappa_{E})$ is the image of the composition
$$
K^\times \too H^1(K,\bmu_{p^k}) \too H^1(K,E[p^k])
$$
where the first map is the classical Kummer map, and the second is induced by 
the inclusion $\tau_{E/K} :\bmu_{p^k} \hookto E[p^k]$.
\end{lem}

\begin{proof}
This follows directly from the commutativity of the diagram
$$
\xymatrix{
K^\times \ar[r]\ar@{->>}_-{\tau_{E/K}}[d] & H^1(K,\bmu_{p^k}) \ar^-{\tau_{E/K}}[d] \\
E(K) \ar^-{\kappa_{E}}[r] & H^1(K,E[p^k]).
}
$$
\end{proof}

\begin{defn}
Suppose $L$ is a quadratic extension of $K$.  Let $N_{L/K} : L^\times \to K^\times$ be the norm map, 
and let $\psi \in \Xset(K)$ be the quadratic character attached to $L/K$.  
Let $\bmu_{p^k} \otimes \psi$ denote a copy of $\bmu_{p^k}$ with $G_K$ acting by $\epsilon_{p^k} \psi$ 
instead of $\epsilon_{p^k}$, and define a twisted Kummer map
$$
\nu_{L/K} : \ker(N_{L/K}) \too H^1(K,\bmu_{p^k} \otimes \psi)
$$
where for $x \in \ker(N_{L/K})$, we choose $u \in \bar{K}$ such that $u^{p^k} = x$, and then 
$\nu_{L/K}(x)$ is represented by the cocycle 
$\sigma \mapsto (u^\sigma)^{\psi(\sigma)}/u$.
We leave it as an exercise to show that $\sigma \mapsto (u^\sigma)^{\psi(\sigma)}/u$ is indeed a cocycle 
with values in $\bmu_{p^k} \otimes \psi$, when $x \in \ker(N_{L/K})$.
\end{defn}

\begin{lem}
\label{3.6}
Suppose $E/K$ has potentially multiplicative reduction, with nontrivial splitting character 
$\psi$ (i.e., $E/K$ does not have split multiplicative reduction).  
Let $q \in K^\times$ be the Tate parameter, and let 
$L$ be the quadratic extension of $K$ attached to $\psi$.  
Let $\H$ be the image of the composition
$$
\ker(N_{L/K}) \map{\nu_{L/K}} H^1(K,\bmu_{p^k}\otimes\psi) \map{\tau_{E/K}} H^1(K,E[p^k]).
$$
\begin{enumerate}
\item
If $p>2$, or if $p=2$ and $q \notin N_{L/K}L^\times$, then $\image(\kappa_{E}) = \H$.
\item
Suppose $p=2$ and  $q = N_{L/K}(\beta)$ with $\beta \in L^\times$.  Then $\image(\kappa_{E})$ is the 
subgroup of $H^1(K,E[p^k])$ generated by $\H$ and the class of the cocycle
$$
\sigma \mapsto \tau_{E/K}((\alpha^\sigma)^{\psi(\sigma)}/\alpha)
$$
where $\alpha \in \bar{K}^\times$ and $\alpha^{2^k} = \beta$.
\end{enumerate}
\end{lem}

\begin{proof}
Consider the commutative diagram with exact rows
$$
\xymatrix{
1 \ar[r] & q^\Z \ar_{2}[d] \ar[r] & L^\times \ar^{N_{L/K}}[d] \ar[r] & L^\times/q^\Z \ar^{\cN}[d] \ar[r] & 1 \\
1 \ar[r] & q^\Z \ar[r] & K^\times \ar[r] & K^\times/q^\Z \ar[r] & 1
}
$$
where we denote by $\cN$ the right-hand norm map.  The snake lemma shows that 
the map $\ker(N_{L/K}) \to \ker(\cN)$ is surjective if $q \notin N_{L/K}L^\times$, 
and if $q = N_{L/K}(\beta)$ then the cokernel of that map has order $2$, and 
$\ker(\cN)$ is generated by the image of $\ker(N_{L/K})$ and $\beta$.

Proposition \ref{tateunif} shows that we have a commutative diagram
$$
\xymatrix{
\ker(N_{L/K}) \ar^-{\nu_{L/K}}[r]\ar_-{\tau_{E/K}}[d] 
   & H^1(K,\bmu_{p^k}\otimes \psi) \ar^-{\tau_{E/K}}[d] \\
E(K) \ar^-{\kappa_{E}}[r] & H^1(K,E[p^k]),
}
$$
and $\tau_{E/K} : \ker(\cN) \to E(K)$ is an isomorphism.  
(To see this last fact, note that $\tau_{E/K} : L^\times/q^\Z \to E(L)$ is a group isomorphism 
satisfying $\tau_{E/K}(u^\sigma) = -\tau_{E/K}(u)^\sigma$ if $\sigma$ is the 
nontrivial automorphism of $L/K$, so $\tau_{E/K}(u) \in E(K)$ if and only if 
$u^\sigma = u^{-1}$, i.e., if and only if $\cN(u) = 1$.)
It follows that $\image(\kappa_{E}) = \H$ if $p > 2$ or if $q \notin N_{L/K}L^\times$, 
and if $q = N_{L/K}(\beta)$ then $\image(\kappa_{E})$ is generated by 
$\H$ and $\kappa_E(\tau_{E/K}(\beta))$.

Suppose $p=2$ and  $q = N_{L/K}(\beta)$, and choose 
$\alpha \in \bar{K}^\times$ with $\alpha^{2^k} = \beta$.  Then 
by definition $\kappa_E(\tau_{E/K}(\beta))$ is represented by the cocycle 
$\sigma \mapsto \tau_{E/K}(\alpha)^\sigma - \tau_{E/K}(\alpha)$, and by Proposition \ref{tateunif} 
$$
\tau_{E/K}(\alpha)^\sigma - \tau_{E/K}(\alpha) = \psi(\sigma)\tau_{E/K}(\alpha^\sigma) - \tau_{E/K}(\alpha) 
   = \tau_{E/K}((\alpha^\sigma)^{\psi(\sigma)}/\alpha).
$$
This completes the proof of the lemma.
\end{proof}

\begin{lem}
\label{3.7}
Let $m = 2^{k+1}$ if $p = 2$, and $m = p^k$ if $p > 2$.
Suppose $E_1$ and $E_2$ are elliptic curves over $K$ with potentially multiplicative reduction, 
and there is a $G_K$-isomorphism $E_1[m] \cong E_2[m]$ that identifies 
$\cC_{E_1/K}[m]$ with $\cC_{E_2/K}[m]$.  
\begin{enumerate}
\item
The splitting characters of $E_1$ and $E_2$ are equal.
\item
Suppose $p = 2$.  
Let $q_i$ be the Tate parameter of $E_i$, and $\pi_1 \in \bar{K}^\times$ with $\pi_1^{2^k} = q_1$.  
Then there is an $x \in K^\times$ and an odd integer $n$ such that 
$q_2^n = q_1 x^m$ and $\tau_{E_1/K}(\pi_1) = \tau_{E_2/K}(\pi_1 x^{2})$ in $E_1[2^k] = E_2[2^k]$.
\end{enumerate}
\end{lem}

\begin{proof}
Since $\cC_{E_1/K}[m]$ and $\cC_{E_1/K}[m]$ are isomorphic $G_K$-modules, 
and $m > 2$, assertion (i) follows directly from Lemma \ref{Cchar}.

Suppose $p=2$, and let $\lambda : E_1[m] \to E_2[m]$ be the $G_K$-isomorphism.  
We will abbreviate $\tau_i := \tau_{E_i/K}$.  Replacing $\lambda$ by an odd multiple 
of $\lambda$ if necessary, we may assume that the diagram
\begin{equation}
\label{sl}
\raisebox{23pt}{
\xymatrix@R4pt@C50pt{
& E_1[m] \ar^-{\lambda}[dd] \\
\bmu_m \ar^{\tau_1}[ur] \ar_{\tau_2}[dr] \\
& E_2[m]
}}
\end{equation}
commutes.

Fix $\beta_1 \in \bar{K}^\times$ such that $\beta_1^m = q_1$.  Then $\tau_1(\beta_1) \in E_1[m]$, 
and we fix $\beta_2 \in \bar{K}^\times$ such that $\tau_2(\beta_2) = \lambda(\tau_1(\beta_1))$. 
Since $\tau_2(\beta_2) \in E_2[m]$, we have $\beta_2^m \in q_2^\Z$, say $\beta_2^m = q_2^n$.  
If $n$ is even, then $\beta_2^{m/2} = \pm q_2^{n/2}$, so 
$\tau_2(\beta_2^{m/2}) = \tau_2(\pm1) \in \cC_{E_2/K}[m]$.  
But $\tau_1(\beta_1^{m/2}) \notin \cC_{E_1/K}[m]$, so this is impossible and 
$n$ must be odd.

Let $\psi$ denote the common splitting character of $E_1$ and $E_2$.  By Proposition \ref{tateunif}, 
for every $\sigma \in G_K$ we have 
$$
\tau_i(\beta_i^\sigma/\beta_i) = \tau_i(\beta_i^\sigma) - \tau_i(\beta_i) 
   = \psi(\sigma)\tau_i(\beta_i)^\sigma - \tau_i(\beta_i).
$$
Applying $\lambda$ we see that $\lambda(\tau_1(\beta_1^\sigma/\beta_1)) = \tau_2(\beta_2^\sigma/\beta_2)$. 
But $\beta_i^\sigma/\beta_i \in \bmu_m$, so by \eqref{sl} we have 
$\beta_1^\sigma/\beta_1 = \beta_2^\sigma/\beta_2$.  Thus $(\beta_2/\beta_1)^\sigma = \beta_2/\beta_1$ 
for every $\sigma \in G_K$, so $\beta_2/\beta_1 \in K$.  

Let $x = \beta_2/\beta_1$.  Then $q_2^n = q_1 x^m$, and if we take $\pi_1 = \beta_1^2$ 
then $\pi_1^{2^k} = q_1$ and
$$
\tau_2(\pi_1x^{2}) = \tau_2(\pi_1\beta_2^{2}/\beta_1^{2}) = \tau_2(\beta_2^{2}) 
   = \lambda(\tau_1(\beta_1^2)) = \lambda(\tau_1(\pi_1)).
$$
This proves (ii).
\end{proof}

\subsection{The group scheme kernel of $p^k$}
Let $\O$ be the ring of integers of $K$ and $\F$ the residue field.  Let $\cE$ be the N\'eron model of $E$, 
so $\cE$ is a smooth group scheme over $\O$.
Let $\cE^0$ be its connected component at the identity, so $\cE^0$ is an open subgroup scheme of $\cE$.  
Let $\cE^0_{/\F} \subset \cE_{/\F}$   be the closed fibers of these group schemes, and 
$\Phi = \cE_{/\F}/\cE^0_{/\F}$ the group of components of the closed fiber, viewed as a (finite {\'e}tale) 
group scheme over ${\bf F}$.
Let $\cE[p^k]$ the kernel of multiplication by $p^k$ on $\cE$; 
this is a group scheme over $\O$ with finite generic fiber.

\begin{prop}
\label{gs}
Suppose that multiplication by $p$ induces a faithfully flat morphism $\cE \to \cE$ over $\O$.
Then $\cE[p^k]$ is a quasi-finite flat group scheme over $\O$, and 
$\image(\kappa_E)$ is the image of the composition
$$
\Hfppf^1(\Spec(\O),\cE[p^k]) \too \Hfppf^1(\Spec(K),\cE[p^k]) \isom H^1(K,E[p^k])
$$
where $\Hfppf^1$ means cohomology of abelian group schemes, computed in the $\mathrm{fppf}$ topology.
\end{prop}

\begin{proof}
This follows from \cite[Lemma 5.1(iii)]{rpav}.  The long exact cohomology sequence attached to 
$0 \to \cE[p^k] \to \cE \map{p^k} \cE \to 0$ gives
\begin{multline*}
\too \Hfppf^0(\Spec(\O),\cE) \map{\;p^k\;} \Hfppf^0(\Spec(\O),\cE) \\
   \too \Hfppf^1(\Spec(\O),\cE[p^k]) \too \Hfppf^1(\Spec(\O),\cE)[p^k].
\end{multline*}
We have $\Hfppf^0(\Spec(\O),\cE) = E(K)$, and $\Hfppf^1(\Spec(\O),\cE)[p^k] = H^1(\F,\Phi)[p^k]$ by 
\cite[Lemma 5.1(iii)]{rpav}.  It follows from our assumption on $p$ that $\Phi[p^k] = 0$, so the upper left vertical 
map in the following commutative diagram is an isomorphism:
$$
\xymatrix@R20pt{
\Hfppf^0(\Spec(\O),\cE)/p^k\Hfppf^0(\Spec(\O),\cE) \ar^-{\cong}[r]\ar_{\cong}[d] & E(K)/p^kE(K) \ar^{\kappa_E}[dd] \\
\Hfppf^1(\Spec(\O),\cE[p^k]) \ar[d] \\
\Hfppf^1(\Spec(K),\cE[p^k]) \ar[r] & H^1(K,E[p^k]).
}
$$
The proposition follows.
\end{proof}

\begin{rem}
If $K$ has residue characteristic $p$, then multiplication by $p$ induces a 
faithfully flat endomorphism of $\cE$ if and only if $E$ has good or multiplicative 
reduction and $p$ does not divide the order of $\Phi$. 
\end{rem}

\section{Proof of Theorem \ref{comp}}
\label{proofsect}

The following is a slightly stronger version of Theorem \ref{comp}.  
We will prove Theorem \ref{comp2} below, and then derive Theorem \ref{comp} from it.
Fix a prime power $p^k$ and a number field $K$.  Recall that for every field $F$, 
we let $\Xset(F) = \Hom(G_F,\{\pm1\})$.

\begin{thm}
\label{comp2}
Suppose $E_1$ and $E_2$ are elliptic curves over $K$.  Let 
$S_i$ be the set of primes of $K$ where $E_i$ has 
potentially multiplicative reduction.
Let $m = p^{k+1}$ if $p \le 3$, and $m = p^k$ if $p>3$.  
Suppose further that:
\begin{enumerate}
\item
there is a $G_K$-isomorphism $E_1[m] \cong E_2[m]$,
\item
$S_1 = S_2$, 
\item
for all $\l \in S_1 = S_2$, the isomorphism of (i) sends $\cC_{E_1/K_\l}[m]$ to $\cC_{E_2/K_\l}[m]$,
\item
for every prime $v$ above $p$, either 
\begin{itemize}
\item
$E_1$ and $E_2$ have potentially multiplicative reduction at $v$, or
\item
$p > 2$, $E_1$ and $E_2$ have good reduction at $v$, and the isomorphism of (i) 
extends to an isomorphism $\cE_1[p^k] \cong \cE_2[p^k]$ over $\Spec(\O_{K_v})$, 
where the group scheme $\cE_i[p^k]$ is the kernel of multiplication by $p^k$ in the 
N\'eron model of $E_i$.
\end{itemize}
\end{enumerate}
Then for every finite extension $F$ of $K$, and every $\chi \in \Xset(F)$, 
there is a canonical isomorphism
$$
\Sel_{p^k}(E_1^\chi/F) \cong \Sel_{p^k}(E_2^\chi/F).
$$
In particular $E_1$ and $E_2$ are $p^k$-Selmer companions over every number field containing $K$. 
\end{thm}

\begin{defn}
\label{seldef}
The {\em $p^k$-Selmer group} $\Sel_{p^k}(E/K)\subset H^1(K,E[p^k])$ of an elliptic curve $E$ over $K$ is 
$$
\Sel_{p^k}(E/K) := \ker\bigl(H^1(K,E[p^k]) \too \dirsum{v}H^1(K_v,E[p^k])/\image(\kappa_{E/K_v})\bigr).
$$
\end{defn}

A $G_K$-isomorphism $E_1[p^k] \cong E_2[p^k]$ allows us to identify 
$$
\text{$H^1(K,E_1[p^k]) = H^1(K,E_2[p^k])$, \; $H^1(K_v,E_1[p^k]) = H^1(K_v,E_2[p^k])$ for every $v$}.
$$ 
We will show that under the hypotheses of Theorems \ref{comp} and \ref{comp2}, 
with these identifications, for every $v$ and for every $\chi\in\Xset(K_v)$, 
the images of the horizontal Kummer maps
\begin{equation}
\label{kd}
\raisebox{15pt}{
\xymatrix@R=10pt{
E_1^\chi(K_v) \ar^-{\kappa_{E_1^\chi/K_v}}[r] &  H^1(K_v,E_1[p^k]) \ar@{=}[d] \\
E_2^\chi(K_v) \ar^-{\kappa_{E_2^\chi/K_v}}[r] &  H^1(K_v,E_2[p^k])
}}
\end{equation}
are equal.  It will then follow from the definition 
that $\Sel_{p^k}(E_1^\chi/K) = \Sel_{p^k}(E_2^\chi/K)$ inside $H^1(K,E_1[p^k]) = H^1(K,E_2[p^k])$ 
for every $\chi$.  This gives the canonical isomorphisms referred to in 
Theorems \ref{comp} and \ref{comp2}.

\begin{proof}[Proof of Theorem \ref{comp2}]
Suppose $E_1$ and $E_2$ satisfy the hypotheses of Theorem \ref{comp2}.
We will show that $\image(\kappa_{E_1^\chi/K_v}) = \image(\kappa_{E_2^\chi/K_v})$ for every 
place $v$ of $K$ and every $\chi \in \Xset(K)$.  
Note that except in the case where $v \mid p$ and $E_1, E_2$ have good reduction at $p$ 
(Case 5 below), 
if the pair $(E_1, E_2)$ satisfies the hypotheses of Theorem \ref{comp2}, then so does the 
pair $(E_1^\chi, E_2^\chi)$ for every $\chi\in\Xset(K)$.  Thus in Cases 1-4, it is enough to consider 
$\chi = 1$.  We will identify $E_1[p^k]$ with $E_2[p^k]$ using the given 
isomorphism, and denote this $G_K$-module simply by $E[p^k]$.

We split the proof into several cases.

\medskip\noindent{\em Case 1: $v \mid \infty$.}
By Lemma \ref{12.2}, $\image(\kappa_{E_i/K_v})$ is zero if $p > 2$, and 
is completely determined by the $G_K$-module 
$E_i[2^{k+1}]$ if $p = 2$.  But if $p=2$ then by assumption $E_1[2^{k+1}] \cong E_2[2^{k+1}]$ as $G_K$-modules, 
so $\image(\kappa_{E_1/K_v}) = \image(\kappa_{E_2/K_v})$.

\medskip\noindent{\em Case 2: $v \nmid \infty$ and $E_1$ has bad, 
potentially good reduction at $v$.} 
By our assumptions, $v \nmid p$ and $E_2$ also has potentially good reduction at $v$. 
By Lemma \ref{samered}, the reduction of $E_2$ is also bad.  
By Lemma \ref{12.2}, we again see that 
$\image(\kappa_{E_i/K_v})$ is zero if $p > 3$, and in general 
is completely determined by the $G_K$-module 
$E_i[m]$, so $\image(\kappa_{E_1/K_v}) = \image(\kappa_{E_2/K_v})$.

\medskip\noindent{\em Case 3: $v \nmid p\infty$ and $E_1$ has good reduction at $v$.} 
By our assumptions, $E_2$ has potentially good reduction at $v$. 
By Lemma \ref{samered}, the reduction of $E_2$ is also good.
It follows that for $i = 1, 2$, the image of $\kappa_{E_i/K_v}$ is the unramified subgroup 
of $H^1(K_v,E[p^k])$ (see \cite[Lemma 4.1]{cassels}), i.e.,
$$
\image(\kappa_{E_i/K_v}) = \ker\bigl(H^1(K_v,E[p^k]) \to H^1(K_v^\ur,E[p^k])\bigr)
$$
where $K_v^\ur$ is the maximal unramified extension of $K_v$.  Since this subgroup 
is independent of $i$, we have $\image(\kappa_{E_1/K_v}) = \image(\kappa_{E_2/K_v})$.

\medskip\noindent{\em Case 4: $v \nmid \infty$ and $E_1$ has potentially multiplicative 
reduction at $v$.} 
By our assumptions, $E_2$ also has potentially multiplicative reduction at $v$, 
and the images of $H^1(K_v,\bmu_{p^k}) \to H^1(K_v,E_i[p^k])$ are identified 
by the isomorphism $E_i[p^k] \cong E_2[p^k]$.

Suppose first that $E_1$ has split multiplicative reduction.  Then so does $E_2$ 
(by Lemma \ref{3.7}(i)), and then Lemma \ref{3.4} shows that 
$\image(\kappa_{E_1/K_v}) = \image(\kappa_{E_2/K_v})$.

Now suppose $E_1$ does not have split multiplicative reduction, and let $L$ be the 
quadratic extension corresponding to the splitting character of $E_1/K_v$.  
By Lemma \ref{3.7}(i), the splitting characters of $E_1$ and $E_2$ are equal.  
Let $q_i$ be the Tate parameter of $E_i$.
By Lemma \ref{3.7}(ii), if $p = 2$ we have $q_1 \in N_{L/K_v}L^\times \iff q_2 \in N_{L/K_v}L^\times$.

If $p > 2$, or if $p = 2$ and $q_1, q_2 \notin N_{L/K_v}L^\times$, then it follows from 
Lemma \ref{3.6}(i) that $\image(\kappa_{E_1/K_v}) = \image(\kappa_{E_2/K_v})$.

Finally, suppose that $p = 2$ and $q_1 \in N_{L/K_v}L^\times$, say $q_1 = N_{L/K_v}(\beta_1)$.  
Let $x\in K_v^\times$ and the odd integer $n$ be as in Lemma \ref{3.7}(ii), so $q_2^n = q_1x^m$, and 
we set $\beta_2 = (\beta_1x^{2^k})^n$ so $N_{L/K_v}(\beta_2) = q_2^n$.  
Fix $\alpha_1 \in \bar{K}_v^\times$ with $\alpha_1^{2^k} = \beta_1$, and set $\alpha_2 = (\alpha_1x)^n$, 
so $\alpha_2^{2^k} = \beta_2$.
Using Lemma \ref{3.6}(ii), to show that $\image(\kappa_{E_1/K_v}) = \image(\kappa_{E_2/K_v})$ 
it is enough to show that the two cocycles
$$
c_i : \sigma \mapsto \tau_{E_i/K_v}((\alpha_i^\sigma)^{\psi(\sigma)}/\alpha_i)
$$
generate the same subgroup of $H^1(K_v,E[2^k])$, for $i = 1, 2$.

If $\sigma$ fixes $L$ (i.e., $\psi(\sigma) = 1$), then 
$((\alpha_i^\sigma)^{\psi(\sigma)}/\alpha_i)^{2^k} = \beta_i^\sigma/\beta_i = 1$, so 
$$
(\alpha_2^\sigma)^{\psi(\sigma)}/\alpha_2 = \alpha_2^\sigma/\alpha_2
   = ((x^\sigma \alpha_1^\sigma)/(x\alpha_1))^n 
   = (\alpha_1^\sigma/\alpha_1)^n = ((\alpha_1^\sigma)^{\psi(\sigma)}/\alpha_1)^n \in \bmu_{2^k},
$$
since $x \in K_v$.  By hypothesis (iii) we have 
$$
\tau_{E_1/K_v}(\bmu_{2^k}) = \cC_{E_1/K_v}[2^k] = \cC_{E_2/K_v}[2^k] = \tau_{E_2/K_v}(\bmu_{2^k})
$$ 
inside $E_1[2^k] = E_2[2^k]$, and $n$ is odd, so $c_1$ and $c_2$ generate the same subgroup of $H^1(K_v,E[2^k])$.

If $\sigma$ does not fix $L$ (i.e., $\psi(\sigma) = -1$), then 
$(\alpha_i^\sigma)^{\psi(\sigma)}/\alpha_i = 1/(\alpha_i^\sigma\alpha_i)$, and 
$(\alpha_1^\sigma\alpha_1)^{2^k} = \beta_1^\sigma\beta_1 = N_{L/K_v}(\beta_1) = q_1$, 
so by Lemma \ref{3.7}(ii) applied with $\pi_1 = \alpha_i^\sigma\alpha_i$ we have
\begin{multline*}
\tau_{E_1/K_v}((\alpha_1^\sigma)^{\psi(\sigma)}/\alpha_1) 
   = \tau_{E_2/K_v}(x^{-2} (\alpha_1^\sigma)^{\psi(\sigma)}/\alpha_1) \\
   = \tau_{E_2/K_v}(((\alpha_2^\sigma)^{\psi(\sigma)}/\alpha_2)^n)
      = n\tau_{E_2/K_v}((\alpha_2^\sigma)^{\psi(\sigma)}/\alpha_2). 
\end{multline*}
Thus $c_1 = nc_2$ in $H^1(K_v,E[2^k])$, and $n$ is odd, so $c_1$ and $c_2$ generate the same subgroup.
Thus  $\image(\kappa_{E_1/K_v}) = \image(\kappa_{E_2/K_v})$ in Case 4.

\medskip\noindent{\em Case 5: $v \mid p$, $p > 2$, $E_1$ and $E_2$ have good reduction at $v$, 
and the isomorphism $E_1[p^k] \cong E_2[p^k]$ extends to an isomorphism $\cE_1[p^k] \cong \cE_2[p^k]$.} 
In this case Proposition \ref{gs} shows that $\image(\kappa_{E_1/K_v}) = \image(\kappa_{E_2/K_v})$.
More generally, suppose $\chi \in \Xset(K_v)$, and let $L$ be the quadratic extension of $K_v$ 
cut out by $\chi$.  Then Proposition \ref{gs} shows that $\image(\kappa_{E_1/L}) = \image(\kappa_{E_2/L})$, 
and this isomorphism preserves the natural action of $\Gal(L/K_v)$.
Let $\sigma$ denote the nontrivial element of $\Gal(L/K_v)$, and for 
every $\Zp[\Gal(L/K_v)]$-module $M$, define $M^- := \{m \in M : m^\sigma = -m\}$.
Since $p > 2$, 
one verifies easily using the Hochschild-Serre spectral sequence that for $i = 1, 2$ we have
\begin{gather*}
E_i^\chi(K_v)/p^k E_i^\chi(K_v) = (E_i(L)/p^k E_i(L))^- , \quad H^1(K_v,E_i^\chi[p^k]) = H^1(L,E_i[p^k])^-, \\
   \image(\kappa_{E_i^\chi/K_v}) = \image(\kappa_{E_i/L})^-
\end{gather*}
and we conclude that $\image(\kappa_{E_1^\chi/K_v}) = \image(\kappa_{E_2^\chi/K_v})$.

\medskip
Thus $\image(\kappa_{E_1^\chi/K_v}) = \image(\kappa_{E_2^\chi/K_v})$ in \eqref{kd} 
for every $v$ and every $\chi$, 
and we conclude that $E_1$ and $E_2$ are $p$-Selmer companions over $K$.  

If $F$ is a finite extension of $K$, then $E_1$ and $E_2$ satisfy the 
hypotheses of Theorem \ref{comp2} over $F$ as well.  Applying the proof above with $F$ 
in place of $K$ shows that $E_1$ and $E_2$ are $p$-Selmer companions over $F$.
\end{proof}

\begin{proof}[Proof of Theorem \ref{comp}]
Theorem \ref{comp} is identical to Theorem \ref{comp2} except for the second part of hypothesis (iv).
Suppose $k=1$, $\p \mid p$, $E_1, E_2$ have good reduction at $\p$, and the ramification 
$e(\p/p) < p-1$ (so in particular $p > 2$).  By Raynaud's theorem 
\cite[\S3.5.5]{raynaud}, the group scheme $\cE_i[p]$ is determined by the Galois module $E_i[p]$.  
Hence in this case the isomorphism $E_1[p] \cong E_2[p]$ necessarily extends to an isomorphism 
$\cE_1[p] \cong \cE_2[p]$.  Thus  
if hypotheses (i) and (iv) of Theorem \ref{comp} are satisfied, then hypothesis (iv) of Theorem \ref{comp2} 
is satisfied.   
In this way Theorem \ref{comp} follows from Theorem \ref{comp2}. 
\end{proof}

\section{Additional questions and remarks}
\label{converse}

In this section we discuss some related questions, including partial converses 
to Theorem \ref{comp}.  Fix a number field $K$.

\subsection{The number of companions}

\begin{prop}
\label{justfin}
Fix an elliptic curve $E_1$ over $K$, and a prime power $p^k > 2$.  There are only finitely many 
elliptic curves $E_2/K$ (up to isomorphism) satisfying hypotheses (i) and (ii) of Theorem \ref{comp}.
\end{prop}

\begin{proof}
Let 
$$
\Sigma = \{\text{primes $\l$ of $K$ : $E_1$ has bad reduction at $\l$}\} \cup \{\l : \l \mid p\}.
$$ 
Suppose that $E_1$ and $E_2$ satisfy (i) and (ii) of Theorem \ref{comp}, and   
$\l$ is a prime of $K$ not dividing $p$.  If $E_2$ has potentially multiplicative reduction at $\l$, 
then $\l \in S_2 = S_1 \subset \Sigma$.
If $E_2$ has additive, potentially good reduction at $\l$, then so does $E_1$ by Lemma \ref{samered}, 
so $\l \in \Sigma$.  Thus $E_2$ has good reduction at all primes outside of $\Sigma$.  
By the ``Shafarevich Conjecture'' (proved by Faltings), there are only finitely many such 
elliptic curves $E_2/K$.
\end{proof}

\subsection{Selmer parity companions}

\begin{defn}
If $p$ is a prime, 
we will say that two elliptic curves $E_1, E_2$ over $K$ are 
{\em $p$-Selmer parity companions} if for every quadratic character $\chi$ of $K$ 
we have
$$
\dim_{\F_p}\Sel_p(E_1^\chi/K) \equiv \dim_{\F_p}\Sel_p(E_2^\chi/K) \pmod{2}.
$$ 
\end{defn}

\begin{rem}
Clearly, if $E_1$ and $E_2$ are $p$-Selmer companions, then they are also $p$-Selmer parity companions.
\end{rem}

\begin{defn}
For every elliptic curve $E/K$, define a set of rational primes
$$
T(E) := \{p :  \text{$E^\chi(K)[p] \ne 0$ for at least one quadratic character $\chi$ of $K$}\}.
$$
\end{defn}

\begin{lem}
\label{fint}
For every elliptic curve $E/K$, the set $T(E)$ is finite.
\end{lem}

\begin{proof}
This is \cite[Proposition 1]{gouvea-mazur} when $K = \Q$, and \cite[Lemma 5.5]{MRstablerank} in general.
\end{proof}

Recall that the Shafarevich-Tate conjecture asserts that the Shafarevich-Tate groups of all 
elliptic curves over $K$ are finite.

\begin{prop}
\label{indofp}
Suppose that either $K = \Q$, or the Shafarevich-Tate conjecture holds for $K$.  
If $E_1$, $E_2$ are elliptic curves over $K$, 
then $E_1$ and $E_2$ are $p$-Selmer parity companions for one prime $p \notin T(E_1) \cup T(E_2)$ 
if and only if they are $p$-Selmer parity companions for every prime $p \notin T(E_1) \cup T(E_2)$.
\end{prop}

\begin{proof}
Suppose first that $K = \Q$, and for every elliptic curve $E/\Q$ define 
$$
\Sel_{p^\infty}(E/\Q) := \varinjlim \Sel_{p^k}(E/\Q) \quad\text{and}\quad r_p(E) := \cork_{\Zp}\Sel_{p^\infty}(E/\Q).
$$  
Let $D$ denote the maximal divisible subgroup of $\Sel_{p^\infty}(E/\Q)$.  
If $E(K)[p] = 0$, then $\Sel_{p}(E/\Q) = \Sel_{p^\infty}(E/\Q)[p]$, and so 
the Cassels pairing shows that
\begin{align*}
\dim_{\Fp}\Sel_{p}(E/\Q) &= r_p(E) + \dim_{\Fp}(\Sel_{p^\infty}(E/\Q)/D)[p]
   \equiv r_p(E) \pmod{2}.
\end{align*}
It is proved in \cite{dokdok} that $(-1)^{r_p(E)} = w(E)$, the global root number of 
the $L$-function of $E$.  
Thus the parity of $\dim_{\Fp}\Sel_{p}(E/\Q)$ is independent of $p \notin T(E)$ in this case.

Now let $K$ be arbitrary.  If the Shafarevich-Tate group of an elliptic curve $E$ over $K$ 
is finite, then the Cassels pairing shows that 
$$
\dim_{\F_p}\Sel_p(E/K) \equiv \rk(E(K)) + \dim_{\Fp}E(K)[p] \pmod{2}.
$$
Therefore in this case also, we have that the parity of $\dim_{\F_p}\Sel_p(E/K)$ 
is independent of $p \notin T(E)$.  The proposition follows easily.
\end{proof}

\begin{exa}
Consider the curves $E_1 = 1026N1$ and $E_2 = 1026O1$ 
over $\Q$ of the Introduction.  Then $E_1$ and $E_2$ are $7$-Selmer companions, 
so they are $7$-Selmer parity companions.  

We claim that $T(E_1)$ is empty and $T(E_2) = \{3\}$.  
Let $\rho_{i,p} : G_\Q \to \Aut(E_i[p])$ be the mod $p$ representation attached to $E_i$.
For every $p \le 4861$ one can check using Sage \cite{sage} 
that $\rho_{i,p}$ is surjective except when $p = 3$ and $i = 2$, 
and for $p > 4861$ we have that $\rho_{i,p}$ is surjective by \cite[Theorem 2]{cojocaru}.  
When $\rho_{i,p}$ is surjective, no quadratic twist $E_i^\chi$ will have a rational point 
of order $p$ because the mod $p$ representation attached to $E_i^\chi$ is $\rho_{i,p} \otimes \chi$, 
which is irreducible.  This proves the claim.

Thus Proposition \ref{indofp} shows that $E_1$ and $E_2$ are $p$-Selmer parity companions for 
every prime $p \ne 3$.  In fact, they are {\em not} $3$-Selmer parity companions 
because $\Sel_3(E_1/\Q) = 0$ and $\dim_{\F_3}\Sel_3(E_2/\Q) = 1$ (note that $E_2$ has a rational 
point of order $3$).

This shows that $p$-Selmer parity companions need not have the 
same Galois action on their $p$-torsion points.  
\end{exa}

We can use the methods and results of \cite{kmr} (see also \cite{kramer}) 
to determine when two elliptic curves are $2$-Selmer parity 
companions.  

For every elliptic curve $E/K$ and place $v$ of $K$, let $\omega_{E,v} : \Xset(K_v) \to \{\pm 1\}$ 
be the map (of sets, not in general a homomorphism) given by \cite[Definition  6.1]{kmr}.

\begin{thm}
\label{one}
Suppose $E_1$, $E_2$ are elliptic curves over $K$, and $\Sigma$ is a finite set of 
places of $K$ containing all places above $2$ and $\infty$, and all primes where either 
$E_1$ or $E_2$ has bad reduction.  Then $E_1$ and $E_2$ are 2-Selmer parity companions if
and only if $\dim_{\F_2}\Sel_2(E_1/K) \equiv \dim_{\F_2}\Sel_2(E_2/K) \pmod{2}$ and
$\omega_{E_1,v} = \omega_{E_2,v}$ for every $v \in \Sigma$.
\end{thm}

\begin{proof}
This follows directly from \cite[Proposition 6.2]{kmr}.
\end{proof}

The next corollary comes close to showing that hypothesis (ii) of Theorem \ref{comp} is 
a necessary condition.

\begin{cor}
\label{four}
Suppose $E_1$ and $E_2$ are 2-Selmer parity companions, and $\l$ is a prime of $K$ not dividing $2$.  
Then: 
\begin{enumerate}
\item
$E_1$ has split multiplicative reduction at $\l$ if and only if $E_2$ does,
\item
$E_1$ has potentially multiplicative reduction at $\l$ if and only if $E_2$ does.
\end{enumerate}
In particular if $S_i$ is the set of primes where $E_i$ has potentially multiplicative 
reduction as in Theorem \ref{comp}, and $\Sigma_2 := \{\l : \l \mid 2\}$, 
then $S_1 \cup \Sigma_2 = S_2 \cup \Sigma_2$.
\end{cor}

\begin{proof}
By Theorem \ref{one}, we have $\omega_{E_1,\l} = \omega_{E_2,\l}$.  
Both assertions now follow directly from \cite[Proposition 6.9]{kmr}.  
\end{proof}

\begin{cor}
Suppose $E_1$ and $E_2$ are 2-Selmer parity companions, and $E_1[4] \cong E_2[4]$ 
as $G_K$-modules.  
If $\l$ is a prime of $K$ not dividing $2$, then 
$E_1$ has good (resp., additive, resp., multiplicative) reduction at $\l$ if and only if $E_2$ does.
\end{cor}

\begin{proof}
Combining Lemma \ref{samered} and Corollary \ref{four}, we get that
\begin{itemize}
\item
$E_1$ has good reduction at $\l$ if and only if $E_2$ does,
\item
$E_1$ has additive, potentially good reduction at $\l$ if and only if $E_2$ does,
\item
$E_1$ has potentially multiplicative reduction at $\l$ if and only if $E_2$ does.
\end{itemize}
In the last case, Lemma \ref{3.7} shows that $E_1$ and $E_2$ have the same splitting character over $K_\l$, 
so $E_1$ has multiplicative reduction at $\l$ if and only if $E_2$ does.  
This completes the proof.
\end{proof}

\subsection{Selmer near-companions}

\begin{defn}
\label{ncdef}
We will say that two elliptic curves $E_1, E_2$ over $K$ are {\em $n$-Selmer near-companions} if 
there is a constant $C = C(E_1,E_2,K)$ such that 
for every $\chi \in \Xset(K)$ there is an abelian group $A_\chi$ and homomorphisms 
$\Sel_n(E_1^\chi/K) \to A_\chi$ and $\Sel_n(E_2^\chi/K) \to A_\chi$ with kernel and cokernel 
of order at most $C$.
\end{defn}

Note that $E_1$ and $E_2$ are $n$-Selmer companions if and only if Definition \ref{ncdef}
is satisfied with $C = 1$.

\begin{thm}
\label{ncomp}
Suppose $E_1$ and $E_2$ are elliptic curves over $K$.  
Let $m = p^{k+1}$ if $p \le 3$, and $m = p^k$ if $p>3$.  
Suppose further that there is a $G_K$-isomorphism $E_1[m] \cong E_2[m]$.  
Then $E_1$ and $E_2$ are $p^k$-Selmer near-companions over every finite extension $F$ of $K$.  
\end{thm}

\begin{proof}
If $E$ is an elliptic curve over $K$ and $\Sigma$ is a finite set of places of $K$, define
$$
\Sel_{p^k}^\Sigma(E/K) := 
   \ker\bigl(H^1(K,E[p^k]) \too \dirsum{v\notin\Sigma}H^1(K_v,E[p^k])/\image(\kappa_{E/K_v})\bigr).
$$
Then the definitions yield an exact sequence
\begin{equation}
\label{achi}
0 \too \Sel_{p^k}(E/K) \too \Sel_{p^k}^\Sigma(E/K) 
   \too \dirsum{v\in\Sigma}H^1(K_v,E[p^k])/\image(\kappa_{E/K_v}).
\end{equation}

Now let $\Sigma$ be the finite set of places of $K$ dividing $n \Delta_1\Delta_2\infty$, where $\Delta_i$ 
is the discriminant of $E_i$.  
Using the given $G_K$-isomorphism to identify $E_1[p^k]$ with $E_2[p^k]$, Cases 2 and 3 of the proof of 
Theorem \ref{comp2} show that if $v \notin \Sigma$, then 
$\image(\kappa_{E_1^\chi/K_v}) = \image(\kappa_{E_2^\chi/K_v})$ for every $\chi \in \Xset(K)$.  
Therefore $\Sel_{p^k}^\Sigma(E_1^\chi/K) = \Sel_{p^k}^\Sigma(E_2^\chi/K)$ inside 
$H^1(K,E_1[p^k]) = H^2(K,E[p^k])$, and we define $A_\chi$ to be this common group.
Now if we let 
$$
C := \prod_{v \in \Sigma} \sup_{\chi\in\Xset(K_v)}|H^1(K_v,E_1^\chi[p^k])| 
$$
then \eqref{achi} applied 
to $E_1^\chi$ and $E_2^\chi$ shows that $E_1$ and $E_2$ are $p^k$-Selmer near-compan\-ions over $K$.
The same proof applies if we replace $K$ by any finite extension.
\end{proof}

\begin{exa}
Take $K = \Q$ and $p^k = 2$, and let $E_1$ and $E_2$ be the elliptic curves  
$26A1$ and $598B1$ in \cite{cremona}:
\begin{align*}
E_1 :& ~y^2 + xy + y = x^3 - 5x - 8, \\
E_2 :& ~y^2 + xy = x^3 - x^2 + 44x + 496.
\end{align*}
The method of \S\ref{method1} below allows us to verify that there is a $G_\Q$-isomorphism 
$E_1[4] \cong E_2[4]$, so $E_1$ and $E_2$ are $2$-Selmer near-companions by Theorem \ref{ncomp}.
Using a little more care, we can show that $\Sel_{p^k}^\Sigma(E_1^\chi/K) = \Sel_{p^k}^\Sigma(E_2^\chi/K)$ 
for every $\chi$, with $\Sigma = \{23\}$.  
In the exact sequences \eqref{achi} for $E_1^\chi$ and $E_2^\chi$ with this $\Sigma$, 
the right hand group has order
$[E_1(\Q_{23}):2 E_1(\Q_{23})] = 2$, independent of $\chi$.  It follows that for every quadratic character 
$\chi$ of $G_\Q$ we have 
$$
\dim_{\F_2}\Sel_2(E_1^\chi/\Q) - \dim_{\F_2}\Sel_2(E_2^\chi/\Q) \in \{-1,0,1\}.
$$
All three values $-1,0,1$ occur.
(Note that $E_1$ has good reduction at $23$ and $E_2$ has multiplicative reduction, so 
hypothesis (ii) of Theorem \ref{comp} does not hold.)
\end{exa}

It is natural to make the following conjecture.

\begin{conj}
If $n \ge 1$ and $E_1, E_2$ are $n$-Selmer near-companions over $K$, then $E_1[n]$ is 
$G_K$-isomorphic to $E_2[n]$.
\end{conj}

\appendix
\section{Checking that $E_1[m] \cong E_2[m]$}
\label{sm8pf}

Suppose $E_1$ and $E_2$ are elliptic curves over $K$.  
In this appendix we discuss two methods for verifying that there is a $G_K$-isomorphism 
$E_1[m] \cong E_2[m]$.

\subsection{Universal families}
\label{method1}

If $m = 3$, $4$, or $5$ 
then \cite[Theorem 4.1]{modp}, \cite[Theorem 4.1]{mod4}, and \cite[Theorem 5.1]{modp}, respectively, 
give explicit models
\begin{equation}
\label{family}
\E_t : y^2 = x^3 + a(t)x + b(t), \qquad a(t), b(t) \in K[t]
\end{equation}
for the family of all elliptic curves $E/K$ 
with $E[m]$ symplectically $G_K$-isomorphic to $E_1[m]$.  
(A symplectic isomorphism is one that preserves the Weil pairings.)
In other words, there is a symplectic $G_K$-isomorphism $E_1[m] \cong E_2[m]$  
if and only if there is an $s \in K$ such that the specialization $\E_s$ is isomorphic over 
$K$ to $E_2$.

To test this, we can simply compute the $j$-invariant of $\E_t$
$$
\J(t) := 1728 \frac{a(t)^3}{4a(t)^3+27 b(t)^2} \in K(t),
$$
and then find (the finite set of) 
all zeros in $K$ of the rational function $\J(t) - j(E_2)$.  
For each zero $s \in K$, we have that $\E_s$ is an elliptic curve over $K$ 
with $j(\E_s) = \J(s) = j(E_2)$, and it is then a simple matter to 
test whether $\E_s$ is isomorphic to $E_2$ over $K$.  If it is, then 
$E_1[m]$ is (symplectically) $G_K$-isomorphic to $E_2[m]$.  But if for every zero $s \in K$ of 
$\J(t) - j(E_2)$ we have that $\E_s$ is not isomorphic over $K$ to $E_2$, then there is 
no symplectic $G_K$-isomorphism $E_1[m] \cong E_2[m]$.

In the case $m=4$, one can show that there is a 
$G_K$-isomorphism $E_1[4] \cong E_1^{\Delta}[4]$ that is {\em not} symplectic, where $E_1^\Delta$ is the quadratic twist of 
$E_1$ by its discriminant $\Delta$.  It follows that if $E_2[4]$ is $G_K$-isomorphic to 
$E_1[4]$, then $E_2[4]$ is symplectically $G_K$-isomorphic either to 
$E_1[4]$ or to $E_1^{\Delta}[4]$, so we can use the argument above to test for all 
$G_K$-isomorphisms, not just the symplectic ones.

\begin{proof}[Proof of Proposition \ref{samemod5}]
Let $E_1$ and $E_2$ be as in Example \ref{ex11}.  
Then $E_1$ has the short Weierstrass model $y^2 = x^3 - 5616x - 494208$.
Applying \cite[Theorem 5.1]{modp} with $a = -5616$ and $b = -494208$ gives an explicit 
model of the elliptic curve $\E_t/\Q(t)$ of \eqref{family}.  
If we set $J(t) := j(\E_t)$ then we compute that the only rational zero of 
$J(t) - j(E_2)$ is $-9/22$.  Specializing $\E_t$ at 
$t = -9/22$ gives a Weierstrass model of an elliptic curve isomorphic over $\Q$ to $E_2$.
Thus there is a (symplectic) $G_\Q$-isomorphism $E_1[5] \cong E_2[5]$.
\end{proof}

\subsection{Constructing an explicit isomorphism}
\label{method2}
Suppose $m = p^n$.  Let $f_i(x) \in K[x]$ be the polynomial of degree $d := (p^{2n} - p^{2n-2})/2$ 
whose roots are the $x$-coordinates of the points in $E_i[p^n]-E_i[p^{n-1}]$.
Suppose further that $f_1$ and $f_2$ are irreducible; this is equivalent to requiring that 
$G_K$ acts transitively on $E_i[p^n]/\{\pm 1\}$ for $i = 1, 2$.

Fix a root $\alpha_1$ of $f_1$.  Suppose that $f_2$ also has at 
least one root $\alpha_2 \in K(\alpha_1)$.  (If $E_1[p^n] \cong E_2[p^n]$, then, this will necessarily 
be the case.)  
Fix such a root and call it $\alpha_2$.  Then 
$K(\alpha_2) = K(\alpha_1)$, so there is a unique polynomial $\phi(x) \in K[x]$ 
of degree less than $d$ such that $\phi(\alpha_1) = \alpha_2$.
Since $\alpha_1$ is a root of $f_2 \circ \phi$, we have that $f_1$ divides $f_2 \circ \phi$, 
so $\phi$ maps all roots of $f_1$ to roots of $f_2$.

Fix a prime $\l$ of $K$ such that $E_1$ and $E_2$ have good reduction at $\l$, and 
$\l$ splits completely in $K(E_1[p^n])/K$.  Fix a prime of $K(E_1[p^n])$ above $\l$.  
If $\F_\l$ denotes the residue field of $\l$, this choice gives us a reduction isomorphism 
$$
\pi : E_1(\bar{K})[p^n] \isom E_1(\F_\l)[p^n].
$$
Note that $\phi$ also maps all roots of $f_1$ in $\F_\l$ to roots of $f_2$ in $\F_l$.  

Fix a basis $P_1, P_2$ of $E_1(\F_\l)[p^n]$.  Let $Q_1, Q_2 \in E_1(\F_\l)[p^n]$ 
be points such that $x(Q_i) = \phi(x(P_i))$.  Define a group homomorphism 
$\varphi : E_1(\F_\l)[p^n] \to E_2(\F_\l)[p^n]$ by 
$$
\text{$\varphi(a P_1 + b P_2) = a Q_1 + b Q_2$ \quad for $a, b \in \Z/p^n\Z$}.
$$
Using the reduction isomorphism $\pi$, we can lift $\varphi$ to a group homomorphism 
$E_1[p^n] \to E_2[p^n]$, which we also denote by $\varphi$.  
Consider the diagram
\begin{equation}
\label{five}
\raisebox{60pt}{\xymatrix@C=20pt{
E_1[p^n] - E_1[p^{n-1}] \ar^{\varphi}[r]\ar_{\pi}^{\cong}[d] & E_2[p^n] - E_2[p^{n-1}]\ar_{\pi}^{\cong}[d] \\
E_1(\F_\l)[p^n] - E_1(\F_\l)[p^{n-1}] \ar^{\varphi}[r]\ar[d] & E_2(\F_\l)[p^n] - E_2(\F_\l)[p^{n-1}]\ar[d] \\
(E_1(\F_\l)[p^n] - E_1(\F_\l)[p^{n-1}])/\{\pm1\} \ar^{[\phi]}[r] \ar_{\pi^{-1}}^{\cong}[d] 
   & (E_2(\F_\l)[p^n] - E_2(\F_\l)[p^{n-1}])/\{\pm1\} \ar_{\pi^{-1}}^{\cong}[d] \\
(E_1[p^n] - E_1[p^{n-1}])/\{\pm1\} \ar^{[\phi]}[r] & (E_2[p^n] - E_2[p^{n-1}])/\{\pm1\}
}}
\end{equation}
where $[\phi]$ denotes the map induced by applying the polynomial $\phi$ to the $x$-coordi\-nates.  
The upper and lower squares are commutative, with vertical isomorphisms. 

Let $S$ be the set of places of $K$ where at least one of $E_1[p^n], E_2[p^n]$ is ramified, 
and $\Xset(K,S) \subset \Xset(K)$ the (finite) subgroup of characters unramified outside of $S$. 

\begin{lem}
\label{A.1}
Suppose that the center square of \eqref{five} is commutative.  Then: 
\begin{enumerate}
\item
$\varphi$ is a group isomorphism,
\item
there is a quadratic character $\psi\in\Xset(K,S)$ such that
$\varphi(P^\sigma) = \psi(\sigma) \varphi(P)^\sigma$ for every $P \in E_1[p^n]$ and $\sigma \in G_K$.
\end{enumerate}
\end{lem}

\begin{proof}
If the center square of \eqref{five} is commutative, then the entire diagram is commutative.
Since $\phi \in K[x]$, the bottom map $[\phi]$ is $G_K$-equivariant.  
Since $f_2$ is irreducible, $G_K$ acts transitively on $(E_2[p^n] - E_2[p^{n-1}])/\{\pm1\}$. 
Hence the bottom map $[\phi]$ is surjective, so the top map $\phi$ is surjective, and (i) follows.  

It also follows from the commutativity of \eqref{five} and the $G_K$-equivariance of $[\phi]$ that 
\begin{equation}
\label{seven}
\text{$\varphi(P^\sigma) = \pm \varphi(P)^\sigma$ \quad for every $P \in E_1[p^n]$ and $\sigma \in G_K$.}
\end{equation}
For every $\sigma \in G_K$, the sets 
$$
\{P \in E_1[p^n] : \varphi(P^\sigma) = \varphi(P)^\sigma\}, \quad
   \{P \in E_1[p^n] : \varphi(P^\sigma) = -\varphi(P)^\sigma\}
$$
are subgroups of $E[p^n]$.  It follows from \eqref{seven} that the union of these two subgroups is 
$E[p^n]$, and therefore one of them must be all of $E[p^n]$.  Thus for every $\sigma$ we can define 
$\psi(\sigma) = \pm1$ so that $\varphi(P^\sigma) = \psi(\sigma) \varphi(P)^\sigma$ for every $P \in E_1[p^n]$.
One sees easily that $\psi\in\Xset(K)$, and $\psi$ is necessarily unramified outside $S$, so $\psi\in\Xset(K,S)$.
\end{proof}

Suppose now that the center square of \eqref{five} is commutative.  Let $\psi \in \Xset(K,S)$ 
be the character of Lemma \ref{A.1}(ii).  Then by Lemma \ref{A.1}, 
$\varphi$ induces an isomorphism $E_1[p^n] \isom E_2^\psi[p^n]$, 
where $E_2^\psi$ is the quadratic twist of $E_2$ by $\psi$.  
We would like to verify that $\psi$ must be the trivial character.

Fix a basis $\chi_1, \ldots, \chi_t$ of the $\F_2$-vector space $\Xset(K,S)$.  
Suppose that for every $i$, $1 \le i \le t$, we can find a prime $\q_i$ of $K$, 
$\q_i \notin S$, such that 
\begin{itemize}
\item
$\chi_i(\q_i) = -1$, 
$\chi_j(\q_i) = 1$ if $j \ne i$, 
\item
the traces of Frobenius of $\q_i$ on $E_1[p^n]$ and $E_2[p^n]$ satisfy
$$
\Tr(\Frob_{\q_i} | E_1[p^n]) \ne -\Tr(\Frob_{\q_i} | E_2[p^n]).
$$
\end{itemize}
Choose a nontrivial character $\chi = \prod_i \chi_i^{a_i} \in \Xset(K,S)$, $a_i \in \{0,1\}$.  
If $a_i \ne 0$ for some $i$, then by our choice of $\q_i$ we have  
$\Tr(\Frob_{\q_i} | E_1[p^n]) \ne \Tr(\Frob_{\q_i} | E_2^\chi[p^n])$, 
so $E_1[p^n] \not\cong E_2^\chi[p^n]$.  Hence $\psi$ must be the trivial character.

All of the steps above can be handled by either Sage \cite{sage} or PARI/GP \cite{pari}.  
Computing the polynomial $\phi(x)$ is the only significantly time-consuming step.  Finding a prime 
$\l$, checking the commutativity of \eqref{five}, and finding primes $\q_i$ with appropriate 
traces of Frobenius is very quick.  (Note that the points $Q_1, Q_2 \in E_1(\F_\l)[p^n]$ are only 
defined up to multiplication by $\pm1$.  If the first choice does not lead to commutativity 
in \eqref{five}, then replacing $Q_1$ by $-Q_1$ may still work.)

\begin{proof}[Proof of Proposition \ref{samemod8}]
Let $E_1$ and $E_2$ be as in Example \ref{ex10}, both of conductor $1242 = 2 \cdot 3^3 \cdot 23$, 
and $m = 8$.  
PARI/GP computes the polynomial $\phi(x)$ of degree $(8^2-4^2)/2 = 24$ in less than a minute 
on a modern desktop computer.  
We take $\l := 19681$, and 
$$
P_1 := (731,4673), P_2 := (3074,1044) \in E_1(\F_{19681}).
$$  
Then $P_1, P_2$ generate $E_1[8]$, and we compute that 
$\phi(731) \equiv 10530 \pmod{19681}$ and  $\phi(3074) \equiv 17962 \pmod{19681}$.  
We define $\varphi$ as above using the points 
$$
Q_1 = (10530,9277), Q_2 = (17962,16270) \in E_2(\F_{19681}),
$$
and we check directly that the center square of \eqref{five} commutes.  Therefore Lemma \ref{A.1}  
shows that $\varphi : E_1[8] \to E_2^\psi[8]$ is a $G_\Q$-equivariant isomorphism, where 
$\psi \in \Xset(\Q,\{\infty, 2, 3, 23\})$.  The group $\Xset(\Q,\{\infty, 2, 3, 23\})$ is generated by 
the characters 
$\chi_{-1}, \chi_2, \chi_{-3}, \chi_{-23}$, where $\chi_d$ is the quadratic character of $\Q(\sqrt{d})/\Q$.
We find the following data, where $a_q(E_i)$ is the trace of Frobenius $\Frob_q$.
$$
\begin{array}{|c|c|c|c|c|c|c|}
\hline
q & \chi_{-1}(q) & \chi_{2}(q) & \chi_{-3}(q) & \chi_{-23}(q) & a_q(E_1) \pmod{8} & a_q(E_2) \pmod{8} \\
\hline
31 & -1 & \phantom{-}1 & \phantom{-}1 & \phantom{-}1 & 2 & 2 \\
\hline
349 & \phantom{-}1 &- 1 & \phantom{-}1 & \phantom{-}1 & 2 & 2 \\
\hline
233 & \phantom{-}1 & \phantom{-}1 & -1 & \phantom{-}1 & 2 & 2 \\
\hline
241 & \phantom{-}1 & \phantom{-}1 & \phantom{-}1 & -1 & 6 & 6 \\
\hline
\end{array}
$$
The argument above shows that $\psi$ is the trivial character, so 
$\varphi : E_1[8] \to E_2[8]$ is a $G_\Q$-equivariant isomorphism.
\end{proof}

\end{document}